\documentclass{article}

\usepackage{amsmath,amssymb,amsthm}

\usepackage{parskip}

\theoremstyle{plain}
\newtheorem*{thm*}{Theorem}
\newtheorem{thm}{Theorem}
\newtheorem*{lemma*}{Lemma}
\newtheorem{lemma}{Lemma}
\newtheorem*{prp*}{Proposition}
\newtheorem{prp}{Proposition}
\newtheorem{crl}{Corollary}
\theoremstyle{remark}
\newtheorem*{remark*}{Remark}

\usepackage[hidelinks]{hyperref}

\usepackage{enumerate}

\usepackage{authblk}

\title{Weak law of large numbers for linear processes}
\author[a]{Vaidotas Characiejus\footnote{Corresponding author.}}
\author[b]{Alfredas Ra\v ckauskas}
\affil[a]{Fakult\"at f\"ur Mathematik, Ruhr-Universit\"at Bochum, Universit\"atsstra{\ss}e 150, 44780~Bochum, Germany\authorcr e-mail: \href{mailto:vaidotas.characiejus@gmail.com}{vaidotas.characiejus@gmail.com}}
\affil[b]{Faculty of Mathematics and Informatics, Vilnius University, Naugarduko g. 24, 03225 Vilnius, Lithuania\authorcr e-mail: \href{mailto:alfredas.rackauskas@mif.vu.lt}{alfredas.rackauskas@mif.vu.lt}}

\begin{document}

\maketitle
\begin{abstract}
\noindent We establish sufficient conditions for the Marcinkiewicz-Zygmund type weak law of large numbers for a linear process $\{X_k:k\in\mathbb Z\}$ defined by $X_k=\sum_{j=0}^\infty\psi_j\varepsilon_{k-j}$ for $k\in\mathbb Z$, where $\{\psi_j:j\in\mathbb Z\}\subset\mathbb R$ and $\{\varepsilon_k:k\in\mathbb Z\}$ are independent and identically distributed random variables  such that $x^p\Pr\{|\varepsilon_0|>x\}\to0$ as $x\to\infty$ with $1<p<2$ and $\operatorname E\varepsilon_0=0$. We use an abstract norming sequence that does not grow faster than $n^{1/p}$ if $\sum|\psi_j|<\infty$. If $\sum|\psi_j|=\infty$, the abstract norming sequence might grow faster than $n^{1/p}$ as we illustrate with an example. Also, we investigate the rate of convergence in the Marcinkiewicz-Zygmund type weak law of large numbers for the linear process.
\let\thefootnote\relax
\let\thefootnote\relax
\end{abstract}

\section{Introduction and the main results}
Suppose that $0<p<2$ and let $\{\xi_k:k\ge1\}$ be independent and identically distributed (i.i.d.) random variables. The Marcinkiewicz-Zygmund type weak law of large numbers (M-Z WLLN) states that
\[
	\frac1{n^{1/p}}\sum_{k=1}^n\xi_k\to c\quad\text{in probability\quad as}\ n\to\infty
\]
if and only if $x^p\Pr\{|\xi_1|>x\}\to0$ as $x\to\infty$ and $c=0$ when $0<p<1$, $\operatorname E[\xi_1I_{\{|\xi_1|\le x\}}]\to c$ as $x\to\infty$ when $p=1$ and $\operatorname E\xi_1=c=0$ when $1<p<2$ (see \cite{feller1971} for $p=1$ and \cite{kallenberg1997} for $p\ne1$). 

We consider a linear process $\{X_k\}=\{X_k:k\in\mathbb Z\}$, i.e.\ random variables defined by
\begin{equation}\label{eq:linearprocess}
	X_k=\sum_{j=0}^\infty\psi_j\varepsilon_{k-j}
\end{equation}
for $k\in\mathbb Z$, where $\{\psi_j\}=\{\psi_j:j\in\mathbb Z\}\subset\mathbb R$ with the convention that $\psi_j=0$ if $j<0$  and $\{\varepsilon_k\}=\{\varepsilon_k:k\in\mathbb Z\}$ are i.i.d.\ random variables. Such process is also called an infinite-order moving average process. Let $\{S_n:n\ge1\}$ be the partial sums of the linear process $\{X_k\}$ given by $S_n=\sum_{k=1}^nX_k$ for $n\ge1$.

We establish sufficient conditions for the M-Z WLLN for the linear process $\{X_k\}$. The motivation for the problems that we investigate comes from the central limit theorem for the linear process $\{X_k\}$.

Suppose for the moment that $\sum\psi_j^2<\infty$, $\operatorname E\varepsilon_0^2<\infty$ and $\operatorname E\varepsilon_0=0$. Ibragimov established that $(\operatorname ES_n^2)^{-1/2}S_n$ converges in distribution to a standard normal random variable as $n\to\infty$ if $\operatorname ES_n^2\to\infty$ as $n\to\infty$ (see Theorem~2.5 of \cite{ibragimov1962} or Theorem~18.6.5 of \cite{ibragimov1971}).  However, the asymptotic behaviour of $\operatorname ES_n^2$ depends on the convergence of $\sum|\psi_j|$. If $\sum|\psi_j|<\infty$ and $\sum_{j=0}^\infty\psi_j\ne0$, then $n^{-1/2}S_n$ converges in distribution to a normal random variable as $n\to\infty$ (see Theorem~2.5 of \cite{ibragimov1962} and Theorem~3.11 of \cite{phillips1992}). If $\sum|\psi_j|=\infty$, the growth rate of the normalizing sequence might be higher than $n^{1/2}$. For example, if $\psi_j=(j+1)^{-d}$ for $j\ge0$ with $1/2<d<1$, then $n^{-(1/2+1-d)}S_n$ converges in distribution to a normal random variable as $n\to\infty$ (see Chapter 3 of \cite{giraitis2012} for more details). Using the same example, we see that the sequence $S_n/n^{1/p}$ does not converge to $0$ in probability for any $p$ such that $p\ge 1/(1/2+1-d)$ despite the fact that $x^p\Pr\{|X_0|>x\}\to0$ as $x\to\infty$ for $p<2$. Hence, the norming sequence $\{n^{1/p}:n\ge 1\}$ is not suitable for the M-Z WLLN for a general linear process.

Let us denote
\begin{equation}\label{eq:W_n(p)}
	\mathcal W_n(p)
	=\bigg(\sum_{j=-\infty}^n|w_{nj}|^p\biggl)^{1/p}
\end{equation}
for $n\ge 1$ and $p\ge1$, where $w_{nj}=\sum_{k=1}^n\psi_{k-j}$. $\mathcal W_n(p)$ is finite for $n\ge1$ and $p\ge1$ if $\sum|\psi_j|^p<\infty$. We now state our main result.
\begin{thm}\label{thm:wlln}
Let $\{\varepsilon_k\}$ be i.i.d.\ random variables such that $x^p\Pr\{|\varepsilon_0|>x\}\to0$ as $x\to\infty$ with $1<p<2$ and $\operatorname E\varepsilon_0=0$. Assume that $\{\psi_j\}\subset\mathbb R$ are such that $\sum|\psi_j|^p<\infty$. If $\mathcal W_n(p)\to\infty$ as $n\to\infty$, then
\[
	\frac{S_n}{\mathcal W_n(p)}\to 0\quad\text{in probability as}\ n\to\infty,
\]
where $\mathcal W_n(p)$ is given by \eqref{eq:W_n(p)} and $S_n$ is the $n$-th partial sum of the linear process $\{X_k\}$ defined by \eqref{eq:linearprocess}.
\end{thm}
There are examples of $\{\psi_j:j\in\mathbb Z\}$ such that $\mathcal W_n(p)$ does not go to infinity as $n\to\infty$ (for instance, $\psi_0=1$, $\psi_1=-1$ and $\psi_j=0$ for $j>1$).

Let us observe that $\mathcal W_n(2)=(\operatorname E\varepsilon_0^2)^{-1/2}\cdot(\operatorname ES_n^2)^{1/2}$ provided that $\sum_{j=0}^\infty\psi_j^2<\infty$, $\operatorname E\varepsilon_0^2<\infty$ and $\operatorname E\varepsilon_0=0$. Hence, the norming sequence $\mathcal W_n(p)$ is essentially an extension of the result of Ibragimov mentioned above to the weak LLN.

Suppose that $\psi_0=1$ and $\psi_j=0$ for $j>0$. Then $X_k=\varepsilon_k$ for $k\in\mathbb Z$ and $\mathcal W_n(p)=n^{1/p}$ for $n\ge1$ and $p>0$. Hence, $x^p\Pr\{|\varepsilon_0|>x\}\to0$ as $x\to\infty$ together with $\operatorname E\varepsilon_0=0$ is not only sufficient, but also necessary in this example. This example shows that the moment assumptions of $\{\varepsilon_k:k\in\mathbb Z\}$ in \autoref{thm:wlln} are sharp. Furthermore, the norming sequence $b_n(p)=n^{1/p}$ for $n\ge1$ with $1<p<2$ is optimal in this example.

Instead of assuming that $x^p\Pr\{|\varepsilon_0|>x\}\to0$ as $x\to\infty$, we can assume that $x^p\Pr\{|X_0|>x\}\to0$ as $x\to\infty$ since these two conditions are equivalent if $1<p<2$, $\operatorname E\varepsilon_0=0$ and $\sum|\psi|^p<\infty$ (see \autoref{prp:equiva} in \autoref{subsec:equiva}).

The growth rate of $\mathcal W_n(p)$ depends on the convergence of $\sum|\psi_j|$. If $\sum|\psi_j|<\infty$, we establish that $\mathcal W_n(p)=O(n^{1/p})$ as $n\to\infty$ (see \autoref{prp:W_n(p)bound} in \autoref{subsec:behaviour}) and obtain the following proposition.
\begin{prp}\label{prp:<}
Let $\{\varepsilon_k\}$ be i.i.d.\ random variables such that $x^p\Pr\{|\varepsilon_0|>x\}\to0$ as $x\to\infty$ with $1<p<2$ and $\operatorname E\varepsilon_0=0$. If $\{\psi_j\}\subset\mathbb R$ are such that $\sum|\psi_j|<\infty$, then
\[
	\frac{S_n}{n^{1/p}}\to 0\quad\text{in probability as}\ n\to\infty,
\]
where $S_n$ is the $n$-th partial sum of the linear process $\{X_k\}$ defined by \eqref{eq:linearprocess}.
\end{prp}

If $0<p\le 1$, $\sum|\psi_j|^p<\infty$, $\operatorname E|\varepsilon_0|^p<\infty$ and $\operatorname E\varepsilon_0=0$ when $p=1$, we have that $n^{-1/p}S_n\to0$ almost surely as $n\to\infty$. The proof of this fact follows from the Marcinkiewicz-Zygmund strong law of large numbers (M-Z SLLN) for i.i.d.\ random variables since the assumption of independence is superfluous when $0<p<1$ (see \cite{marcinkiewicz1937}).  For the case when $p=1$, see Corollary~2.1.3 and Example~2.1.4 of \cite{straumann2005}.

Thus, a linear process $\{X_k\}$ is short-range dependent or has short memory with respect to the M-Z WLLN if $\sum|\psi_j|<\infty$.

If $\sum|\psi_j|=\infty$, the sequence $\mathcal W_n(p)$ might grow faster than $n^{1/p}$.  As an example, we consider $\psi_j=(j+1)^{-d}$ for $j\ge0$ with $1/p<d<1$. Then $\mathcal W_n(p)\sim c\cdot n^{1/p+1-d}$ as $n\to\infty$ with a positive constant $c$ (see \autoref{prp:b_n(p)pw} in \autoref{subsec:behaviour}) and the linear processes $\{X_k\}$ is long-range dependent or has long memory with respect to the M-Z WLLN. We obtain the following corollary of \autoref{thm:wlln}.
\begin{crl}
Let $\{\varepsilon_k\}$ be i.i.d.\ random variables such that $x^p\Pr\{|\varepsilon_0|>x\}\to0$ as $x\to\infty$ with $1<p<2$ and $\operatorname E\varepsilon_0=0$. Suppose that $\{\psi_j\}\subset\mathbb R$ is defined by $\psi_j=(j+1)^{-d}$ for $j\ge0$ with $1/p<d<1$. Then
\[
	\frac{S_n}{n^{1/p+1-d}}\to 0\quad\text{in probability as}\ n\to\infty,
\]
where $S_n$ is the $n$-th partial sum of the linear process $\{X_k\}$ defined by \eqref{eq:linearprocess}.
\end{crl}

In the next proposition, we establish the rate of convergence in the M-Z WLLN for a linear process under stronger assumptions than in \autoref{thm:wlln}.\begin{prp}\label{prp:rate}
Let $\{\varepsilon_k\}$ be i.i.d.\ random variables such that $\operatorname E[|\varepsilon_0|^p\log(1+|\varepsilon_0|)]<\infty$ with $1<p<2$ and $\operatorname E\varepsilon_0=0$. Assume that $\{\psi_j\}\subset\mathbb R$ are such that $\sum|\psi_j|^p<\infty$. If there exists $p<q\le 2$ such that
\begin{equation}\label{eq:condrate}
	\frac{\mathcal W_n(q)}{\mathcal W_n(p)}=O(n^{1/q-1/p})\quad\text{as}\quad n\to\infty,
\end{equation}
then
\[
	\sum_{n=1}^\infty n^{-1}\Pr\{|\mathcal W_n^{-1}(p)S_n|>\delta\}<\infty
\]
for each $\delta>0$, where $\mathcal W_n(p)$ is given by \eqref{eq:W_n(p)} and $S_n$ is the $n$-th partial sum of the linear process $\{X_k\}$ defined by \eqref{eq:linearprocess}.
\end{prp}

As far as we know, there is only one paper about the M-Z SLLN under long-range dependence.  The M-Z SLLN for a particular linear process $\{X_k\}$ is investigated in \cite{louhichi2000}. It is assumed that $\{\varepsilon_k\}$ are i.i.d.\ symmetric $\alpha$-stable random variables with $1<\alpha<2$ and that there exists $1\le s<\alpha$ such that $\sum|\psi_j|^s<\infty$. Let us observe that it suffices to assume that $\sum|\psi_j|^\alpha<\infty$ to define such a linear process, so a stronger assumption about $\{\psi_j:j\ge0\}$ than needed to define a linear process is made in \cite{louhichi2000}. Under these assumptions, it is proved in \cite{louhichi2000} that $n^{-1/p}S_n\to0$ almost surely as $n\to\infty$ for all $p$ such that $1/p>1-1/s+1/\alpha$.  It seems that the proof also works when $\operatorname E|\varepsilon_0|^p<\infty$ for $p<q$ with $1<q<2$ and there exists $1\le s<q$ such that $\sum|\psi_j|^s<\infty$. The M-Z SLLN for linear processes with long range dependence is established in \cite{louhichi2000} under stronger assumptions than ours. We establish the M-Z WLLN, but we make sharp and natural assumptions on $\{\varepsilon_k\}$.

The rest of the paper is organised as follows. In \autoref{sec:pre}, we present some facts about moments of random variables, establish the almost sure convergence of series \eqref{eq:linearprocess} and investigate the asymptotic behaviour of $\mathcal W_n(p)$. The proofs of our main results are in \autoref{sec:proofs}.

\section{Preliminaries}\label{sec:pre}
For $0<p<\infty$,  $L_{p,\infty}$ denotes the space of real valued random variables $\xi$ on $(\Omega,\mathcal F,\Pr)$ such that
\[
	\|\xi\|_{p,\infty}=\bigl(\sup_{x>0}(x^p\Pr\{|\xi|>x\})\bigr)^{1/p}<\infty.
\]
The functional $\|\cdot\|_{p,\infty}$ is a quasi-norm and satisfies the following inequality
\begin{equation}\label{eq:weakmomenttr}
	\|\zeta+\xi\|_{p,\infty}
	\le\max\{2,2^{1/p}\}(\|\zeta\|_{p,\infty}+\|\xi\|_{p,\infty})
\end{equation}
for $\zeta,\xi\in L_{p,\infty}$. We have, when $r>p>0$,
\begin{equation}\label{eq:weakmoment}
	\|\xi\|_{p,\infty}\le(\operatorname E|\xi|^p)^{1/p}\le\Bigl(\frac r{r-p}\Bigr)^{1/p}\|\xi\|_{r,\infty}.
\end{equation}
There exists a constant $C(p)>0$ with $1\le p<2$ such that
\begin{equation}\label{eq:weakmomentsum}
	\Bigl\|\sum_{i=1}^n\xi_i\Bigr\|_{p,\infty}^p\le C(p)\sum_{i=1}^n\|\xi_i\|_{p,\infty}^p
\end{equation}
for independent and \emph{symmetric} random variables $\xi_1,\ldots,\xi_n\in L_{p,\infty}$ (see Proposition 9.13 of \cite{ledoux1991}).
For any $r>0$, any $a>0$ and any random variable $\xi$,
\begin{equation}\label{eq:trmom}
	\operatorname E[|\xi|^rI_{\{|\xi|\le a\}}]
	=r\int_0^ax^{r-1}\Pr\{|\xi|>x\}\mathrm dx-a^r\Pr\{|\xi|>a\}.
\end{equation}

\subsection{Convergence of the series}
We establish sufficient conditions for the almost sure convergence of series \eqref{eq:linearprocess}.
\begin{prp}\label{prp:convergence}
Let $p>0$. Suppose that $\|\varepsilon_0\|_{p,\infty}<\infty$ and~$\operatorname E\varepsilon_0=0$ if $p>1$. Series~\eqref{eq:linearprocess} converges almost surely if:
\begin{enumerate}[(i)]
	\item $p\ne1$, $p\ne 2$ and $\sum_{j=0}^\infty|\psi_j|^q<\infty$, where~$q=\min\{p,2\}$;
	\item $p=1$ and $\sum_{j=0}^\infty|\psi_j|\log|\psi_j|^{-1}<\infty$;
	\item $p=2$ and $\sum_{j=0}^\infty|\psi_j|^2\log|\psi_j|^{-1}<\infty$.
\end{enumerate}
\end{prp}
\begin{proof}
Assume without loss of generality that~$\psi_j\ne0$ for each~$j\ge0$. We establish the convergence of the following series:
\begin{align}
	&\sum_{j=0}^\infty\Pr\{|\varepsilon_0|>|\psi_j|^{-1}\};\label{eq:first}\\
	&\sum_{j=0}^\infty\psi_j\operatorname E[\varepsilon_0I_{\{|\varepsilon_0|\le|\psi_j|^{-1}\}}];\label{eq:second}\\
	&\sum_{j=0}^\infty\psi_j^2\operatorname{Var}[\varepsilon_0I_{\{|\varepsilon_0|\le|\psi_j|^{-1}\}}].	\label{eq:third}
\end{align}

First, we establish convergence of series~\eqref{eq:first}. We have that
\begin{equation}\label{eq:prob}
	\sum_{j=0}^\infty|\psi_j|^p|\psi_j|^{-p}\Pr\{|\varepsilon_0|>|\psi_j|^{-1}\}
	\le\|\varepsilon_0\|_{p,\infty}^p\sum_{j=0}^\infty|\psi_j|^p.
\end{equation}

Secondly, we investigate the convergence of series~\eqref{eq:second}. If $0<p<1$, then
\[
	\sum_{j=0}^\infty|\psi_j|\operatorname E[|\varepsilon_0|I_{\{|\varepsilon_0|\le|\psi_j|^{-1}\}}]
	\le\sum_{j=0}^\infty|\psi_j|\|\varepsilon_0\|_{p,\infty}^p\frac{|\psi_j|^{p-1}}{1-p}
	=\frac{\|\varepsilon_0\|_{p,\infty}^p}{1-p}\sum_{j=0}^\infty|\psi_j|^p
\]
using \eqref{eq:trmom}. If $p=1$, then 
\[
	\int_0^{|\psi_j|^{-1}}\Pr\{|\varepsilon_0|>s\}\mathrm ds\le1+\|\varepsilon_0\|_{1,\infty}\log|\psi_j|^{-1}
\]
for $j\ge J$, where $J\ge0$ is such that $|\psi_j|^{-1}\ge1$ when $j\ge J$. Using \eqref{eq:trmom},
\[
	\sum_{j=J}^\infty|\psi_j|\operatorname E[|\varepsilon_0|I_{\{|\varepsilon_0|\le|\psi_j|^{-1}\}}]
	\le\sum_{j=J}^\infty|\psi_j|+\|\varepsilon_0\|_{1,\infty}\sum_{j=J}^\infty|\psi_j|\log|\psi_j|^{-1}.
\]
If $p>1$, we have that 
\begin{equation}\label{eq:expectation}
		\operatorname E[\varepsilon_0I_{\{|\varepsilon_0|\le|\psi_j|^{-1}\}}]
		=\operatorname E[\varepsilon_0I_{\{|\varepsilon_0|\le|\psi_j|^{-1}\}}]-\operatorname E\varepsilon_0
		=-\operatorname E[\varepsilon_0I_{\{|\varepsilon_0|>|\psi_j|^{-1}\}}]
\end{equation}
and
\begin{align*}
	\operatorname E[|\varepsilon_0|I_{\{|\varepsilon_0|>|\psi_j|^{-1}\}}]
	&=\int_0^{|\psi_j|^{-1}}\Pr\{|\varepsilon_0|I_{\{|\varepsilon_0|>|\psi_j|^{-1}\}}>x\}\mathrm dx\\
	&\qquad+\int_{|\psi_j|^{-1}}^\infty\Pr\{|\varepsilon_0|I_{\{|\varepsilon_0|>|\psi_j|^{-1}\}}>x\}\mathrm dx\\
	&=|\psi_j|^{-1}\Pr\{|\varepsilon_0|>|\psi_j|^{-1}\}+\int_{|\psi_j|^{-1}}^\infty\Pr\{|\varepsilon_0|>x\}\mathrm dx
\end{align*}
since $\Pr\{|\varepsilon_0|I_{\{|\varepsilon_0|>|\psi_j|^{-1}\}}>x\}=\Pr\{|\varepsilon_0|>|\psi_j|^{-1}\}$ for $0\le x\le|\psi_j|^{-1}$. Hence,
\begin{align*}
	\sum_{j=0}^\infty|\psi_j|\operatorname E[|\varepsilon_0|I_{\{|\varepsilon_0|>|\psi_j|^{-1}\}}]
	&\le\sum_{j=0}^\infty|\psi_j|\|\varepsilon_0\|_{p,\infty}^p\Bigl[1+\frac1{p-1}\Bigr]|\psi_j|^{p-1}\\
	&=\|\varepsilon_0\|_{p,\infty}^p\Bigl[1+\frac1{p-1}\Bigr]\sum_{j=0}^\infty|\psi_j|^p.
\end{align*}

Finally, we complete the proof by establishing the convergence of series~\eqref{eq:third}. If $0<p<2$, then
\[
	\sum_{j=0}^\infty|\psi_j|^2\operatorname E[|\varepsilon_0|^2I_{\{|\varepsilon_0|\le|\psi_j|^{-1}\}}]
	\le\sum_{j=0}^\infty|\psi_j|^22\|\varepsilon_0\|_{p,\infty}^p\frac{|\psi_j|^{p-2}}{2-p}
	=\frac{2\|\varepsilon_0\|_{p,\infty}^p}{2-p}\sum_{j=0}^\infty|\psi_j|^p	
\]
using \eqref{eq:trmom}. If $p=2$, then
\[
	\int_0^{|\psi_j|^{-1}}2s\Pr\{|\varepsilon_0|>s\}\mathrm ds
	\le1+2\|\varepsilon_0\|_{2,\infty}^2\log|\psi_j|^{-1}
\]
for $j\ge J$, where $J\ge0$ is such that $|\psi_j|^{-1}\ge1$ when $j\ge J$. Using \eqref{eq:trmom},
\[
	\sum_{j=J}^\infty|\psi_j|^2\operatorname E[|\varepsilon_0|^2I_{\{|\varepsilon_0|\le|\psi_j|^{-1}\}}]
	\le\sum_{j=J}^\infty|\psi_j|^2+2\|\varepsilon_0\|_{2,\infty}^2\sum_{j=J}^\infty|\psi_j|^2\log|\psi_j|^{-1}.
\]
If $p>2$, then $\operatorname{Var}[\varepsilon_0I_{\{|\varepsilon_0|\le|\psi_j|^{-1}\}}]\to\operatorname{Var}\varepsilon_0$ as~$j\to\infty$ and the~series
\[
	\sum_{j=0}^\infty\psi_j^2\operatorname{Var}[\varepsilon_0I_{\{|\varepsilon_0|\le|\psi_j|^{-1}\}}]
\]
converges if~$\sum_{j=0}^\infty\psi_j^2<\infty$.
\end{proof}

\begin{remark*}
Suppose that $\psi_j>0$ for each $j\ge0$ and that $\varepsilon_0$ has the density function
\[
	f(x)=
	\begin{cases}
		\frac2\pi\frac1{1+x^2}&\text{if }x\ge0,\\
		0&\text{if }x<0.
	\end{cases}
\]
Then $\varepsilon_0\in L_{1,\infty}$ and 
\[
	\operatorname E[\varepsilon_0I_{\{\varepsilon_0\le\psi_j^{-1}\}}]
	=\frac2\pi\int_0^{\psi_j^{-1}}\frac x{1+x^2}\mathrm dx
	=\frac1\pi\log(\psi_j^{-2}+1)
\]
so that the series
\[
		\sum_{j=0}^\infty\operatorname E[\psi_j\varepsilon_{k-j}I_{\{|\psi_j\varepsilon_{k-j}|\le1\}}]
\]
converges if and only if $\sum_{j=0}^\infty\psi_j\log\psi_j^{-1}<\infty$ since $\log(\psi_j^{-2}+1)\sim 2\log\psi_j^{-1}$. Hence, the condition in \autoref{prp:convergence} when $p=1$  is sharp. Similarly, we can construct an example to show that the condition in \autoref{prp:convergence} when $p=2$ is also sharp.
\end{remark*}

\subsection{Asymptotic behaviour of $\{\mathcal W_n(p)\}$}\label{subsec:behaviour}
We investigate the asymptotic behaviour of $\{\mathcal W_n(p)\}$ in this subsection.
\begin{prp}\label{prp:W_n(p)bound}
Suppose that $p\ge1$ and $\sum|\psi_j|<\infty$. Then
\[
	\mathcal W_n(p)=O(n^{1/p})
\]
as $n\to\infty$.
\end{prp}
\begin{proof}
We have that  $\mathcal W_n^p(p)=\sum_{j=0}^\infty|\sum_{k=1}^n\psi_{k+j}|^p+\sum_{j=1}^n|\sum_{k=0}^{n-j}\psi_k|^p$. There exists $N\ge0$ such that $\sum_{j=N}^\infty|\psi_j|\le1$ since $\sum|\psi_j|<\infty$. It follows that $|\sum_{k=1}^n\psi_{k+j}|\le 1$ for $j\ge N-1$ and $n\ge1$. Hence,
\[
	\sum_{j=N-1}^\infty\biggl|\sum_{k=1}^n\psi_{k+j}\biggr|^p
	\le\sum_{j=N-1}^\infty\biggl|\sum_{k=1}^n\psi_{k+j}\biggr|
	\le\sum_{k=1}^n\sum_{j=N}^\infty|\psi_{k+j}|
	\le n
\]
and
\[
	\sum_{j=0}^\infty\biggl|\sum_{k=1}^n\psi_{k+j}\biggr|^p
	\le\sum_{j=0}^{N-2}\biggl|\sum_{k=1}^n\psi_{k+j}\biggr|^p+n
	\le(N-1)\biggl(\sum_{k=1}^\infty|\psi_k|\biggr)^p+n.
\]
Also, we have that
\[
	\sum_{j=1}^n\biggl|\sum_{k=0}^{n-j}\psi_k\biggr|^p
	\le\sum_{j=1}^n\bigg(\sum_{k=0}^{n-j}|\psi_k|\biggr)^p
	\le n\bigg(\sum_{k=0}^\infty|\psi_k|\biggr)^p.
\]
The proof is complete.
\end{proof}
\begin{prp}\label{prp:b_n(p)pw}
Let $p>1$ and $1/p<d<1$. Suppose that $\psi_j=(j+1)^{-d}$ for $j\ge0$. Then $\mathcal W_n(p)\sim c\cdot n^{1/p+1-d}$ as $n\to\infty$, where $c$ is a positive constant.
\end{prp}
\begin{proof}
We have that
\[
	\mathcal W_n^p(p)=\sum_{j=1}^\infty\biggl|\sum_{k=1}^n(k+j)^{-d}\biggr|^p+\sum_{j=1}^n\biggl|\sum_{k=1}^{n-j+1}k^{-d}\biggr|^p.
\]
We obtain that the limit
\begin{align*}
&\lim_{n\to\infty}\frac1{n^{1+p(1-d)}}\sum_{j=1}^\infty\biggl|\sum_{k=1}^n(k+j)^{-d}\biggr|^p\\
	&\qquad=\lim_{n\to\infty}\sum_{j=1}^\infty\frac1n\sum_{l=(j-1)n+1}^{jn}\biggl|\frac1n\sum_{k=1}^n\biggl(\frac kn+\frac ln\biggr)^{-d}\biggr|^p\\
	&\qquad=\sum_{j=1}^\infty\lim_{n\to\infty}\frac1n\sum_{l=(j-1)n+1}^{jn}\biggl|\frac1n\sum_{k=1}^n\biggl(\frac kn+\frac ln\biggr)^{-d}\biggr|^p\\
	&\qquad=\sum_{j=1}^\infty\int_{j-1}^j\biggl|\int_0^1(x+y)^{-d}\mathrm dx\biggr|^p\mathrm dy\\
	&\qquad=\int_0^\infty\biggl|\int_0^1(x+y)^{-d}\mathrm dx\biggr|^p\mathrm dy
\end{align*}
is finite and positive. By approximating sums with definite integrals, we obtain $\sum_{k=1}^{n-j+1}k^{-d}\le(n-j+1)^{1-d}/(1-d)$
and
\[
\sum_{j=1}^n\biggl|\frac{(n-j+1)^{1-d}}{1-d}\biggr|^p
	\le\int_0^n\frac{(n-x+1)^{p(1-d)}}{[1-d]^p}\mathrm dx
	=\frac{(n+1)^{1+p(1-d)}-1}{[1+p(1-d)][1-d]^p}.
\]
The proof is complete.
\end{proof}
\begin{prp}\label{prp:sup}
Let $1\le p<2$. Suppose that $\sum|\psi_j|^p<\infty$. If $\mathcal W_n(p)\to\infty$ as $n\to\infty$, then
\[
	\frac{\sup_{j\in\mathbb Z}|w_{nj}|}{\mathcal W_n(p)}\to0\quad\text{as}\quad n\to\infty.
\]
\end{prp}
\begin{proof}
By the mean value theorem, there exists $0<\theta<1$ such that
\[
	(t+h)^p-t^p=p(t+\theta h)^{p-1}h
\]
for $t\ge0$ and $h\ge0$. Hence,
\begin{equation}\label{eq:mvt}
	(t+h)^p\le t^p+p(t+h)^{p-1}h
\end{equation}
for $t\ge0$ and $h\ge0$.

For $m\in\mathbb Z$, $k\in\mathbb Z$ and $n\ge1$, we have that
\begin{align*}
	&\sum_{j=m-k}^m|w_{nj}|^p=\\
	&\le\sum_{j=m-k}^m(|w_{n,j-1}|+|w_{nj}-w_{n,j-1}|)^p\\
	&\le\sum_{j=m-k}^m|w_{n,j-1}|^p+p\sum_{j=m-k}^m(|w_{n,j-1}|+|w_{nj}-w_{n,j-1}|)^{p-1}|w_{nj}-w_{n,j-1}|
\end{align*}
using the triangle inequality and inequality \eqref{eq:mvt}.
 
For $a\ge0$, $b\ge0$ and $p>0$,
\begin{equation}\label{eq:pandr}
	(a+b)^p\le\max\{1,2^{p-1}\}(a^p+b^p).
\end{equation}

By inequality \eqref{eq:pandr},
\begin{multline*}
\sum_{j=m-k}^m(|w_{n,j-1}|+|w_{nj}-w_{n,j-1}|)^{p-1}|w_{nj}-w_{n,j-1}|\\
	\le\sum_{j=m-k}^m|w_{n,j-1}|^{p-1}|w_{nj}-w_{n,j-1}|+\sum_{j=m-k}^m|w_{nj}-w_{n,j-1}|^p
\end{multline*}
since $0\le p-1<1$. We have that
\[
	\sum_{j=m-k}^m|w_{nj}|^p-\sum_{j=m-k}^m|w_{n,j-1}|^p
	=|w_{nm}|^p-|w_{n,m-k-1}|^p
\]
so it follows that
\begin{multline*}
	|w_{nm}|^p\le|w_{n,m-k-1}|^p+p\sum_{j=m-k}^m|w_{n,j-1}|^{p-1}|w_{nj}-w_{n,j-1}|\\
	+p\sum_{j=m-k}^m|w_{nj}-w_{n,j-1}|^p.
\end{multline*}
Letting $k\to\infty$, we obtain
\[
	\sup_{m\in\mathbb Z}|w_{nm}|^p
	\le p\sum_{j=-\infty}^\infty|w_{nj}|^{p-1}|w_{nj}-w_{n,j-1}|+p\sum_{j=-\infty}^\infty|w_{nj}-w_{n,j-1}|^p.
\]
By H\"older's inequality with $p$ and $q=p/(p-1)$,
\begin{multline*}
	\sum_{j=-\infty}^\infty|w_{nj}|^{p-1}|w_{nj}-w_{n,j-1}|\\
	\le\biggl(\sum_{j=-\infty}^\infty|w_{nj}|^p\biggr)^{1/q}\biggl(\sum_{j=-\infty}^\infty|w_{nj}-w_{n,j-1}|^p\biggr)^{1/p}.
\end{multline*}
It follows that
\begin{equation}\label{eq:sup}
	\sup_{m\in\mathbb Z}|w_{nm}|^p
	\le p\biggl(\sum_{j=-\infty}^\infty|w_{nj}|^p\biggr)^{1/q}\biggl(2^p\sum_{j=0}^\infty|\psi_j|^p\biggr)^{1/p}+2^pp\sum_{j=0}^\infty|\psi_j|^p
\end{equation}
using the triangle inequality and inequality \eqref{eq:pandr}
since
\[
	w_{nj}-w_{n,j-1}=\psi_{1-j}-\psi_{n-j+1}.
\]
Inequality \eqref{eq:sup} completes the proof since $\mathcal W_n(p)\to\infty$ as $n\to\infty$ and $\sum|\psi_j|<\infty$. The proof is complete.
\end{proof}

\section{Proofs of the main results}\label{sec:proofs}
\autoref{thm:wlln} and \autoref{prp:<} follows from the next lemma that establishes sufficient conditions for any sequence $\{b_n(p)\}=\{b_n(p):n\ge1\}\subset\mathbb R$ with $1<p<2$ to be the norming sequence in the M-Z WLLN for the linear process $\{X_k\}$.
\begin{lemma}\label{lemma:b_n(p)}
Let $1<p<2$. Suppose that $\sum|\psi_j|^p<\infty$, $x^p\Pr\{|\varepsilon_0|>x\}\to0$ as $x\to\infty$ and $\operatorname E\varepsilon_0=0$. If
\begin{equation}\label{eq:wllncond}
	\lim_{n\to\infty}\frac{\sup_{j\le n}|w_{nj}|}{b_n(p)}=0\qquad\text{and}\qquad W(p)=\sup_{n\ge1}\frac{\mathcal W_n(p)}{b_n(p)}<\infty,
\end{equation}
then
\[
	\frac{S_n}{b_n(p)}\to 0\quad\text{in probability as}\ n\to\infty.
\]
\end{lemma}
We choose $b_n(p)=\mathcal W_n(p)$ for $n\ge1$ and use \autoref{prp:sup} to establish \autoref{thm:wlln}. We set $b_n(p)=n^{1/p}$ for $n\ge1$ and use \autoref{prp:W_n(p)bound} to obtain \autoref{prp:<}.

We make use of the technique of truncation in the proofs. Let us introduce the notations. Suppose that $\{r_{nj}\}=\{r_{nj}:n\ge1,j\in\mathbb Z\}$ are positive real numbers. Denote $\mu_{nj}'=\operatorname E[\varepsilon_0I_{\{|\varepsilon_0|\le r_{nj}\}}]$ and 	$\mu_{nj}''=\operatorname E[\varepsilon_0I_{\{|\varepsilon_0|>r_{nj}\}}]$ for $n\ge 1$ and $j\in\mathbb Z$, where $I_A$ denotes the indicator function of a set $A$. Set $\varepsilon_j=\varepsilon_{nj}'+\varepsilon_{nj}''$, where
\[
\varepsilon_{nj}'=\varepsilon_jI_{\{|\varepsilon_j|\le r_{nj}\}}-\mu_{nj}'\quad\text{and}\quad\varepsilon_{nj}''=\varepsilon_jI_{\{|\varepsilon_j|>r_{nj}\}}-\mu_{nj}''
\]
so that $\operatorname E\varepsilon_0=\operatorname E\varepsilon_{nj}'=\operatorname E\varepsilon_{nj}''=0$ for $n\ge 1$ and $j\in\mathbb Z$. By $\{\tilde\varepsilon_{nj}'':n\ge1,\,j\in\mathbb Z\}$ we denote an idependent copy of $\{\varepsilon_{nj}'':n\ge1,\,j\in\mathbb Z\}$ so that $\{\varepsilon_j''-\tilde\varepsilon_j'':n\ge1,\,j\in\mathbb Z\}$ are independent and symmetric random variables. For $n\ge1$, denote
\begin{equation}\label{eq:truncsums}
	S_n'=\sum_{j=-\infty}^nw_{nj}\varepsilon_{nj}'
	\quad\text{and}\quad	
	S_n''=\sum_{j=-\infty}^nw_{nj}\varepsilon_{nj}''.
\end{equation}
We need the following auxiliary lemma.
\begin{lemma}\label{lemma:weakmoment}
Suppose that $x^p\Pr\{|\varepsilon_0|>x\}\to0$ as $x\to\infty$ for some $p>1$ and $r_n\to\infty$ as $n\to\infty$, where $r_n=\inf_{j\le n}r_{nj}$. Then $\sup_{j\le n}\|\varepsilon_{nj}''\|_{p,\infty}\to 0\quad\text{as}\quad n\to\infty$.
\end{lemma}
\begin{proof}
Let us denote $M_n''=\operatorname E[|\varepsilon_0|I_{\{|\varepsilon_0|>r_n\}}]$. Using the triangle inequality and the fact that $|\mu_{nj}''|\le M_n''$, we obtain
\begin{align*}
\sup_{j\le n}\|\varepsilon_{nj}''\|_{p,\infty}^p
	&=\sup_{j\le n}\sup_{x>0}(x^p\Pr\{|\varepsilon_0I_{\{|\varepsilon_0|>r_{nj}\}}-\mu_{nj}''|>x\})\\
	&\le\sup_{x>0}(x^p\Pr\{|\varepsilon_0|I_{\{|\varepsilon_0|>r_n\}}>x-M_n''\}).
\end{align*}
Observe that
\[
\Pr\{|\varepsilon_0|I_{\{|\varepsilon_0|>r_n\}}>x-M_n''\}=
	\begin{cases}
		1,&\text{if } 0<x<M_n'',\\
		\Pr\{|\varepsilon_0|>r_n\},&\text{if }M_n''\le x\le M_n''+r_n,\\
		\Pr\{|\varepsilon_0|>x\},&\text{if } x>M_n''+r_n.
	\end{cases}
\]
Hence,
\begin{multline*}
	\sup_{x>0}(x^p\Pr\{|\varepsilon_0|I_{\{|\varepsilon_0|>r_n\}}>x-M_n''\})\\
	\le\max\{(M_n'')^p,(M_n''+r_n)^p\Pr\{|\varepsilon_0|>r_n\},\sup_{x>M_n''+r_n}(x^p\Pr\{|\varepsilon_0|>x\})\}\to0
\end{multline*}
as $n\to\infty$ since $x^p\Pr\{|\varepsilon_0|>x\}\to0$ as $x\to\infty$ and $r_n\to\infty$ as $n\to\infty$, so that $\lim_{n\to\infty}M_n''=0$.
\end{proof}

\subsection{Proof of \autoref{lemma:b_n(p)}}
First, we show that $b_n^{-1}(p)S_n''$ converges in probability to $0$ as $n\to\infty$ in \autoref{prp:wlln>} and then we move to the proof of \autoref{lemma:b_n(p)}.
\begin{prp}\label{prp:wlln>}
Let $1<p<2$. If  $x^p\Pr\{|\varepsilon_0|>x\}\to0$ as $n\to\infty$ and $\lim_{n\to\infty}r_n=\infty$, where $r_n=\inf_{j\le n}r_{nj}$, then
\[
	b_n^{-1}(p)S_n''\to0
\]
in probability as $n\to\infty$.
\end{prp}
\begin{proof}
The functional $\|\cdot\|_{p,\infty}$ is only a quasi-norm and it is not necessarily continuous, so that inequality \eqref{eq:weakmomentsum} might not hold for series. Hence, we split the series $\sum_{j=-\infty}^nw_{nj}\varepsilon_{nj}''$ into two parts.
For $N<n$ and $\delta>0$, we have that
\begin{multline}\label{eq:trun>}
	\Pr\{|b_n^{-1}(p)S_n''|>\delta/2\}
	\le\Pr\biggl\{\biggl|b_n^{-1}(p)\sum_{j=-\infty}^Nw_{nj}\varepsilon_j''\biggr|>\delta/4\biggr\}\\
	+\Pr\biggl\{\biggl|b_n^{-1}(p)\sum_{j=N+1}^nw_{nj}\varepsilon_j''\biggr|>\delta/4\biggr\}.
\end{multline}
The series $\sum_{j=-\infty}^nw_{nj}\varepsilon_j''$ converges almost surely for each $n\ge1$. Therefore $b_n^{-1}(p)\sum_{j=-\infty}^Nw_{nj}\varepsilon_j''\to0$ almost surely as $N\to-\infty$ for each $n\ge1$ and there exists $N(n)<n$ for each $n\ge1$ and each $\delta>0$ such that
\[
	\biggl|b_n^{-1}(p)\sum_{j=-\infty}^{N(n)}w_{nj}\varepsilon_j''\biggr|\le\delta/4
\]
almost surely and the first term on the right side of~\eqref{eq:trun>} is $0$.

Using Markov's inequality and the fact that $\operatorname E\tilde\varepsilon_{nj}''=0$ for each~$n\ge1$ and each~$j\in\mathbb Z$, we obtain
\begin{align*}
	\Pr\biggl\{\biggl|b_n^{-1}(p)\sum_{j=N(n)+1}^nw_{nj}\varepsilon_j''\biggr|>\delta/4\biggr\}
	&<\frac4\delta\operatorname E\biggl|b_n^{-1}(p)\sum_{j=N(n)+1}^nw_{nj}(\varepsilon_{nj}''-\operatorname E\tilde\varepsilon_{nj}'')\biggr|\\
	&\le\frac4\delta\operatorname E\biggl|b_n^{-1}(p)\sum_{j=N(n)+1}^nw_{nj}(\varepsilon_{nj}''-\tilde\varepsilon_{nj}'')\biggr|.
\end{align*}
By inequalities~\eqref{eq:weakmoment},\eqref{eq:weakmomentsum} and~\eqref{eq:weakmomenttr},
\begin{align*}
	&\Pr\biggl\{\biggl|b_n^{-1}(p)\sum_{j=N(n)+1}^nw_{nj}\varepsilon_{nj}''\biggr|>\delta/2\biggr\}\\
	&\qquad\le\frac4\delta\Bigl(\frac p{p-1}\Bigr)\biggl\|b_n^{-1}(p)\sum_{j=N(n)+1}^nw_{nj}(\varepsilon_{nj}''-\tilde\varepsilon_{nj}'')\biggr\|_{p,\infty}\\
	&\qquad\le\frac{4C(p)^{1/p}}\delta\Bigl(\frac p{p-1}\Bigr)\cdot b_n^{-1}(p)\biggl(\sum_{j=N(n)+1}^n\|w_{nj}(\varepsilon_{nj}''-\tilde\varepsilon_{nj}'')\|_{p,\infty}^p\biggr)^{1/p}\\
	&\qquad\le\frac{16C(p)^{1/p}}\delta\Bigl(\frac p{p-1}\Bigr)\cdot b_n^{-1}(p)\biggl(\sum_{j=N(n)+1}^n|w_{nj}|^p\|\varepsilon_{nj}''\|_{p,\infty}^p\biggr)^{1/p}.
\end{align*}
Since $(\sum_{j=N(n)+1}^n|w_{nj}|^p)^{1/p}\le\mathcal W_n(p)$, we have that
\begin{multline*}
\Pr\biggl\{\biggl|b_n^{-1}\sum_{j=N(n)+1}^nw_{nj}\varepsilon_{nj}''\biggr|>\delta/2\biggr\}\\
	\le\frac{16C^{1/p}(p)}\delta\Bigl(\frac p{p-1}\Bigr)\frac{\mathcal W_n(p)}{b_n(p)}\Bigl(\sup_{j\le n}\|\varepsilon_{nj}''\|_{p,\infty}^p\Bigr)^{1/p}.
\end{multline*}
$\sup_{j\le n}\|\varepsilon_{nj}''\|_{p,\infty}\to0$ as $n\to\infty $ using \autoref{lemma:weakmoment} since $x^p\Pr\{\|\varepsilon_0\|>x\}\to0$ as $x\to\infty$ and $r_n\to\infty$ as $n\to\infty$. The proof is complete.
\end{proof}

\begin{proof}[Proof of \autoref{lemma:b_n(p)}] For $n\ge1$, $j\in\mathbb Z$, $\tau>0$ and $\delta>0$, set
\begin{equation}\label{eq:r_njnonsummable}
r_{nj}
	=\Bigl[\frac{\tau\delta^2(2-p)}{8\|\varepsilon_0\|_{p,\infty}^pW(p)}\Bigr]^{\frac1{2-p}}\cdot\frac{b_n(p)}{|w_{nj}|}.
\end{equation}
Using Chebyshev's inequality and \eqref{eq:trmom}, we obtain
\[
	\limsup_{n\to\infty}\Pr\{|b_n^{-1}(p)S_n'|>\delta/2\}
	\le4\delta^{-2}\limsup_{n\to\infty}\biggl[b_n^{-2}(p)\sum_{j=-\infty}^n|w_{nj}|^2\operatorname E|\varepsilon_{nj}'|^2\biggr]
	\le\tau.
\]
Since $r_n=\inf_{j\le n}r_{nj}\to\infty$ as $n\to\infty$, where $r_{nj}$ is given by \eqref{eq:r_njnonsummable}, we have that $\Pr\{|b_n^{-1}S_n''|>\delta/2\}\to0$ as $n\to\infty$ for each $\delta>0$ by \autoref{prp:wlln>}. Hence, $\limsup_{n\to\infty}\Pr\{|b_n^{-1}(p)S_n|>\delta\}\le\tau$ for each $\delta>0$ and $\tau>0$. The proof is complete.
\end{proof}

\subsection{Equivalence of moment assumptions}\label{subsec:equiva}
We show that $x^p\Pr\{|\varepsilon_0|>x\}\to0$ as $x\to\infty$ and $x^p\Pr\{|X_0|>x\}\to0$ as $x\to\infty$ are equivalent conditions if $1<p<2$, $\operatorname E\varepsilon_0=0$ and $\sum|\psi|^p<\infty$.
\begin{prp}\label{prp:equiva}
Let $1<p<2$ and suppose that $\operatorname E\varepsilon_0=0$ and $\sum|\psi_j|^p<\infty$. Then $x^p\Pr\{|\varepsilon_0|>x\}\to0$ as $x\to\infty$ if and only if $x^p\Pr\{|X_0|>x\}\to0$ as $x\to\infty$.
\end{prp}
\begin{proof}
We only prove sufficiency. See Lemma 3.7 of \cite{juodis2009} for the proof of necessity.

Suppose that $\{\tilde\varepsilon_k:k\in\mathbb Z\}$ is an independent copy of $\{\varepsilon_k:k\in\mathbb Z\}$ so that $\{\varepsilon_k-\tilde\varepsilon_k:k\in\mathbb Z\}$ are independent and symmetric random variables. Let $a>0$ be such that $\Pr\{|X_0|\le a\}\ge1/2$. We have that
\begin{align*}
\Pr\biggl\{\biggl|\sum_{j=0}^\infty\psi_j(\varepsilon_{-j}-\tilde\varepsilon_{-j})\biggr|>t\biggr\}
	&\ge\Pr\{|X_0|>t+a\}\Pr\{|X_0|\le a\}\\
	&\ge\frac12\Pr\{|X_0|>t+a\}.
\end{align*}
Hence, we assume that $\{\varepsilon_k:k\in\mathbb Z\}$ are symmetric in the remainder of the proof since the general case can be proved using symmetrization.

We have that
\[
	x^p\Pr\{|X_0|>2x\}
	\le x^p\Pr\biggl\{\biggl|\sum_{j=0}^\infty\psi_j\varepsilon_{-j}'\biggr|>x\biggr\}
		+x^p\Pr\biggl\{\biggl|\sum_{j=0}^\infty\psi_j\varepsilon_{-j}''\biggr|>x\biggr\},
\]
where $\varepsilon_j'=\varepsilon_jI_{\{|\varepsilon_j|\le \delta x\}}$ and $\varepsilon_j''=\varepsilon_jI_{\{|\varepsilon_j|>\delta x\}}$.

Using Chebyshev's inequality and \eqref{eq:trmom},
\begin{align*}
	x^p\Pr\biggl\{\biggl|\sum_{j=0}^\infty\psi_j\varepsilon_{-j}'\biggr|>x\biggr\}
	&\le x^{p-2}\operatorname E\biggl|\sum_{j=0}^\infty\psi_j\varepsilon_{-j}'\biggr|^2
	=x^{p-2}\sum_{j=0}^\infty\psi_j^2\operatorname E[\varepsilon_0^2I_{\{|\varepsilon_0|\le \delta x\}}]\\
	&\le \frac{2\|\varepsilon_0\|_{p,\infty}^p}{2-p}\sum_{j=0}^\infty\psi_j^2\delta^{2-p}.
\end{align*}
Since the series $\sum_{j=0}^\infty\psi_j\varepsilon_{-j}''$ converges almost surely for each $x>0$ and $\delta>0$, there exists $J\ge0$ that depends on $\delta$ and $x$ such that $|\sum_{j=J}^\infty\psi_j\varepsilon_{-j}''|\le x/2$ almost surely. Hence,
\begin{align*}
	\Pr\biggl\{\biggl|\sum_{j=0}^\infty\psi_j\varepsilon_{-j}''\biggr|>x\biggr\}
	&\le\Pr\biggl\{\biggl|\sum_{j=0}^J\psi_j\varepsilon_{-j}''\biggr|>x/2\biggr\}+\Pr\biggl\{\biggl|\sum_{j=J+1}^\infty\psi_j\varepsilon_{-j}''\biggr|>x/2\biggr\}\\
	&\le\Pr\biggl\{\biggl|\sum_{j=0}^J\psi_j\varepsilon_{-j}''\biggr|>x/2\biggr\}.
\end{align*}
Using the inequality $\Pr\{|\xi|>x\}\le\|\xi\|_{p,\infty}^p/x^p$ for a random variable $\xi$, $x>0$ and $p>0$ and inequality \eqref{eq:weakmomentsum}, we obtain
\begin{align*}
	x^p\Pr\biggl\{\biggl|\sum_{j=0}^J\psi_j\varepsilon_{-j}''\biggr|>x/2\biggr\}
	&\le 2^p\biggl\|\sum_{j=0}^J\psi_j\varepsilon_{-j}''\biggr\|_{p,\infty}^p
	\le 2^pC(p)\sum_{j=0}^J\|\psi_j\varepsilon_{-j}''\|_{p,\infty}^p\\
	&\le 2^pC(p)\sum_{j=0}^\infty|\psi_j|^p\|\varepsilon_0''\|_{p,\infty}^p.
\end{align*}
We have that $\|\varepsilon_0''\|_{p,\infty}^p=\sup_{x>0}(x^p\Pr\{|\varepsilon_0|>\max\{\delta t,x\}\})\to0$ as $t\to\infty$ and, for $\delta>0$,
\[
	\limsup_{x\to\infty}x^p\Pr\{|X_0|>x\}
	\le\frac{\|\varepsilon_0\|_{p,\infty}^p}{2-p}\sum_{j=0}^\infty\psi_j^2\cdot \delta^{2-p}.\qedhere
\]
\end{proof}

\subsection{Proof of the convergence rate}
In this subsection, we prove \autoref{prp:rate}.
\begin{proof}[Proof of \autoref{prp:rate}]
Set $r_{nj}=n^{1/p}$ for each $n\ge1$ and each $j\in\mathbb Z$. We have that
\begin{multline}\label{eq:sllnseries}
	\sum_{n=1}^\infty n^{-1}\Pr\{|\mathcal W_n^{-1}(p)S_n|>\delta\}
	\le\sum_{n=1}^\infty n^{-1}\Pr\{|\mathcal W_n^{-1}(p)S_n'|>\delta/2\}\\
		+\sum_{k=1}^\infty n^{-1}\Pr\{|\mathcal W_n^{-1}(p)S_n''|>\delta/2\}.
\end{multline}

Using Markov's inequality, the von Bahr-Esseen inequality (see \cite{vonbahr1965}) and~\eqref{eq:condrate}, we obtain
\begin{align*}
	\sum_{n=N}^\infty k^{-1}\Pr\{|\mathcal W_n^{-1}(p)S_n'|>\delta/2\}
	&\le\frac{2^{q+1}}{\delta^q}\sum_{n=N}^\infty n^{-1}\biggl(\frac{\mathcal W_n(q)}{\mathcal W_n(p)}\biggr)^q\operatorname E|\varepsilon_{k0}'|^q\\
	&\le M\frac{2^{q+1}}{\delta^q}\sum_{n=N}^\infty \frac{\operatorname E|\varepsilon_{n0}'|^q}{n^{q/p}}
\end{align*}
for $N\ge1$ such that
\[
\frac{\mathcal W_n(q)}{\mathcal W_n(p)}\le Mn^{1/q-1/p}
\]
for $n\ge N$, where $M$ is a positive constant. The series
\[
	\sum_{n=1}^\infty \frac{\operatorname E|\varepsilon_{n0}'|^q}{n^{q/p}}
\]
converges if and only if $\operatorname E|\varepsilon_0|^p<\infty$. Hence, the first series on the right side of~\eqref{eq:sllnseries} converges.

Using Markov's inequality and the von Bahr-Esseen inequality, we have that
\[
	\sum_{n=1}^\infty n^{-1}\Pr\{|\mathcal W_n^{-1}(p)S_n''|>\delta/2\}
	\le\frac{2^{p+1}}{\delta^p}\sum_{n=1}^\infty n^{-1}\operatorname E|\varepsilon_{n0}''|^p
\]
and
\begin{align*}
 	\sum_{n=1}^\infty n^{-1}\operatorname E|\varepsilon_{n0}''|^p
	&=\sum_{n=1}^\infty n^{-1}\sum_{l=n}^\infty\operatorname E[|\varepsilon_0|^pI_{\{l^{1/p}<|\varepsilon_0|\le(l+1)^{1/p}\}}]\\
	&=\sum_{l=1}^\infty\sum_{n=1}^ln^{-1}\operatorname E[|\varepsilon_0|^pI_{\{l^{1/p}<|\varepsilon_0|\le(l+1)^{1/p}\}}]\\
	&\le\operatorname E|\varepsilon_0|^p+\sum_{l=1}^\infty\log l\operatorname E[|\varepsilon_0|^pI_{\{l^{1/p}<|\varepsilon_0|\le(l+1)^{1/p}\}}]\\
	&\le\operatorname E|\varepsilon_0|^p+p\operatorname E[|\varepsilon_0|^p\log|\varepsilon_0|I_{\{|\varepsilon_0|>1\}}].
\end{align*}
The second series on the right side of~\eqref{eq:sllnseries} also converges. The proof is complete.
\end{proof}

\textbf{Acknowledgement.} The financial support of the DFG (German Science Foundation) SFB 823: Statistical modeling of nonlinear dynamic processes is gratefully acknowledged.

\bibliographystyle{abbrv}
\bibliography{bibliography}
\end{document}